\documentclass{article}

\usepackage[margin=1in]{geometry}

\usepackage{amsfonts,amsmath,amsthm,amscd,amssymb,latexsym}

\usepackage{graphicx}
\usepackage{url}		
\usepackage{ytableau}
\usepackage{tikz}
\usepackage{stackengine}
\usepackage{enumerate}
\usepackage{caption}
\usepackage{subcaption}
\usepackage{hyperref}
\usepackage{color, soul}

\theoremstyle{plain}
\newtheorem{theorem}{Theorem}
\newtheorem{lemma}{Lemma}
\newtheorem{proposition}{Proposition}
\newtheorem{corollary}{Corollary}
\newtheorem{definition}{Definition}
\theoremstyle{definition}
\newtheorem{example}{Example}
\newtheorem{algorithm}{Algorithm}
\begin{document}

\title{Maximum arrangements of nonattacking kings on the $2n\times 2n$ chessboard}
\author{Tricia Muldoon Brown\\
Georgia Southern University}
\date{}

\maketitle

\begin{abstract} 
To count the number of maximum independent arrangements of $n^2$ kings on a $2n\times 2n$ chessboard, we build a $2^n \times (n+1)$ matrix whose entries are independent arrangements of $n$ kings on $2\times 2n$ rectangles.  Utilizing upper and lower bound functions dependent of the entries of the matrix, we recursively construct independent solutions, and provide a straight-forward formula and algorithm.\\

\end{abstract}


\section{Introduction}\label{sec_introduction}

The problem of finding and counting the number of independent, also called nonattacking, arrangements of pieces on a chessboard is a long-established problem in mathematics and computer science.  Many variations have been studied including modifications of traditional pieces and of board size and shape.  In particular, we will be concerned here with maximum independent arrangements on a square chessboard where arrangements of pieces are always fixed, that is, rotations and reflections are considered distinct.   For some traditional chess pieces, bishops, rooks, knights, and pawns, both the size of a maximum independent set and the number of distinct maximal independent arrangements are known.  For the other pieces, kings and queens, only partial results are known.  For those interested in these kinds of questions, books by Dudeney~\cite{Dudeney}, Kraitchik~\cite{Kraitchik}, Madachy~\cite{Madachy}, Watkins~\cite{Watkins}, and Yaglom and Yaglom~\cite{Yaglom_Yaglom} provide good references, and Table~\ref{tab_counts} illustrates the currently known enumerative results.

\begin{table}[ht]
\centering
\scalebox{0.85}{
\begin{tabular}{|c|c|c|c|c|c|c|}
\hline
 & Kings & Queens & Bishops & Knights & Rooks & Pawns\\
\hline
Number of pieces&&&&&&\\
in a maximum & $\big\lceil \frac{n}{2} \big\rceil^2$ & $n$ & $2n-2$ & $\big\lceil \frac{n}{2} \big\rceil^2 + \big\lfloor \frac{n}{2} \big\rfloor ^2$ & $n$ &$n \big\lceil \frac{n}{2} \big\rceil$\\
arrangement&&&&&&\\
\hline
Number of&  1 if $n$ is odd,& unknown if  & & 1 if $n$ is odd, & & 2 if $n$ is odd,\\
maximum&unknown if& $n>27$ & $2^n$ & 2 if $n$ is even & $n!$ & ${n\choose n/2}^2$ if $n$ is even\\
arrangements    &$n>26$ is even & & & & & \\
\hline
\end{tabular}
}
\caption{Enumerative results for maximum independent sets placed on an $n\times n$ chessboard where $n>1$}
\label{tab_counts}
\end{table}

For the still-open questions, it is known that the maximum number of nonattacking queens on an $n\times n$ chessboard is $n$, but counting the number of these arrangements is a famously difficult question as can be the many variations on this problem. (See a survey paper by Bell and Stevens~\cite{Bell_Stevens} on this topic.)  However, here we wish to consider questions of maximum arrangements of nonattacking kings.  First, the maximum size is easily found.  We observe that every $2\times 2$ square on the chessboard may contain at most one king, so the upper bound on the number of kings is $\lceil \frac{n}{2} \rceil ^2$.  But this number can be achieved for both odd- and even-length boards by alternating kings in both rows and columns, as shown in Figure~\ref{fig_kingssolutions}.

\begin{figure}[ht]
\begin{center}
\scalebox{0.9}{\includegraphics{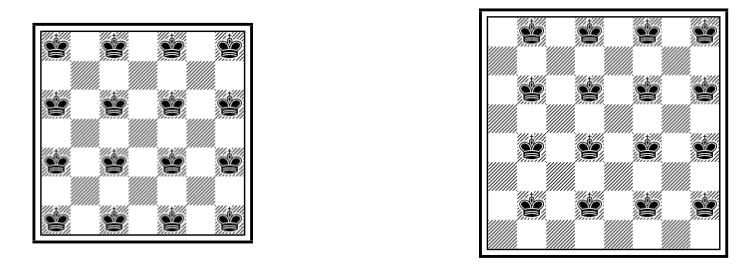}}
\end{center}
\caption{Maximum arrangements of 16 kings on $7\times 7$ and $8\times 8$ chessboards}
\label{fig_kingssolutions}
\end{figure}

Determining the number of distinct maximum independent arrangements in the case that $n$ is odd, is also straight-forward.  The arrangement placing kings in alternating positions and alternating rows as illustrated in Figure~\ref{fig_kingssolutions} is unique.  To see this, observe that every $2\times n$ strip and every $n\times 2$ strip must contain exactly $\lceil \frac{n}{2} \rceil$ kings in order to achieve the maximum.  Because the kings necessarily must at least be placed in alternating rows or columns, the only ways to do this is to place a king in the top row or leftmost column and alternate to the bottom row or rightmost column. Thus kings must appear in all four corners and in alternating rows and columns.

Counting maximum arrangements of independent kings on an even-length $2n\times 2n$ chessboard is not so simple as there is more latitude for king placement in this case.  Figure~\ref{fig_evenkings} illustrates some of these arrangements.  As we will now only be considering even-length boards, note that for the rest of the discussion $n$ will always refer to the half-length of the even-length chessboard.

\begin{figure}[ht]
%
%
\begin{center}
\scalebox{0.9}{\includegraphics{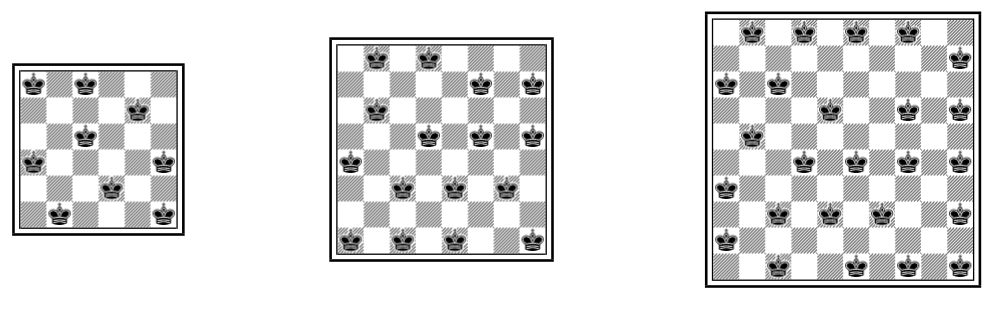}}
\end{center}
\caption{Examples of maximum arrangements of nonattacking kings on chessboards with even-lengths 6, 8, and 10}
\label{fig_evenkings}
\end{figure}

The question of enumerating maximum independent arrangements of kings on the even-length $2n\times 2n$ chessboard has been studied by Knuth~\cite{Knuth} in an unpublished work and later by Wilf~\cite{Wilf}, and then Larsen~\cite{Larsen} who gives an asymptotic approximation.  Entry \href{https://oeis.org/A018807}{A018807} in the Online Encyclopedia of Integer Sequences~\cite{OEIS} contributes counts up to $n=26$, and Kot\v{e}\v{s}ovec~\cite{Kotesovec} also provides enumerative results in his extensive book on independent chessboard arrangements.  Variations have been studied by Calkin et. al.~\cite{Calkin_etal1, Calkin_etal2} who consider all non-maximum arrangements, while Abramson and Moser~\cite{Abramson_Moser} count arrangements of $n$ independent kings with exactly one king in each row and column.

For the original question, Knuth's technique was to partition the chessboard into $2\times 2$ squares in such a way that would generate two vectors of integers.  He then sets up a graph where directed paths with certain restrictions could be used to find maximum independent kings solutions.  Wilf's strategy is to partition the chessboard into $2n\times 2$ columns, but then uses matrix multiplication of the \textit{transfer matrix} to identify solutions.  Algorithmically, this requires construction of the transfer matrix and repeated matrix multiplication on the $2^n (n+1) \times 2^n (n+1)$ matrix.  Here, we will use a smaller $2^n \times (n+1)$ matrix and apply techniques similar to those used in an earlier work by the author~\cite{TMB_pawns} on arrangements of independent pawns to find a formula to enumerate the number of arrangements of $n^2$ kings.  The complexity is still exponential, but smaller with a worst-case estimate of $O(3^n n^3)$.

\section{Counting nonattacking arrangements}
We begin similarly to Wilf~\cite{Wilf} by constructing a matrix which we will call $M_{2n}$ whose entries are all possible arrangements of $n$ independent kings on a $2\times 2n$ rectangle.  First, given such $2\times 2n$ rectangle, we partition the rectangle into $n$ $2\times 2$ squares.  From left to right, index each of these squares with the integers $1, 2, \ldots, n$.  We note, each $2\times 2$ square must contain exactly one king so that the set of kings placed on the rectangle is both non-attacking and maximum.  To determine the placement of the king, we need to know if the king is in the top or bottom row of the square and if the king is in the left or right column.  Thus we associate any $2\times 2n$ arrangement with a subset $A \subseteq [n]$ and an index $1\leq k \leq n+1$ where $A$ equal the set of indices of the $2\times 2$ squares that have their king in the top row of the square and $k$ is the index of the square such that all squares with index less than $k$ have a left king and all squares with an index greater than or equal to $k$ have a right king.  We use the $2^n$ subsets to index the rows of the matrix $M_{2n}$ and the values $1\leq k \leq n+1$ to index the columns.  The $2\times 2n$ rectangles are arranged from left to right in increasing order by the number of kings which appear in the left column of their $2\times 2$ square in the partition of the rectangle.  Specifically, the column index is the index of the leftmost square containing a king in the right column, and also is one plus the number of such kings appearing in a left column.  See Figure~\ref{3matrix} for an example of this matrix in the case $n=3$.

\begin{figure}[ht]
\begin{center}
\scalebox{0.9}{\includegraphics{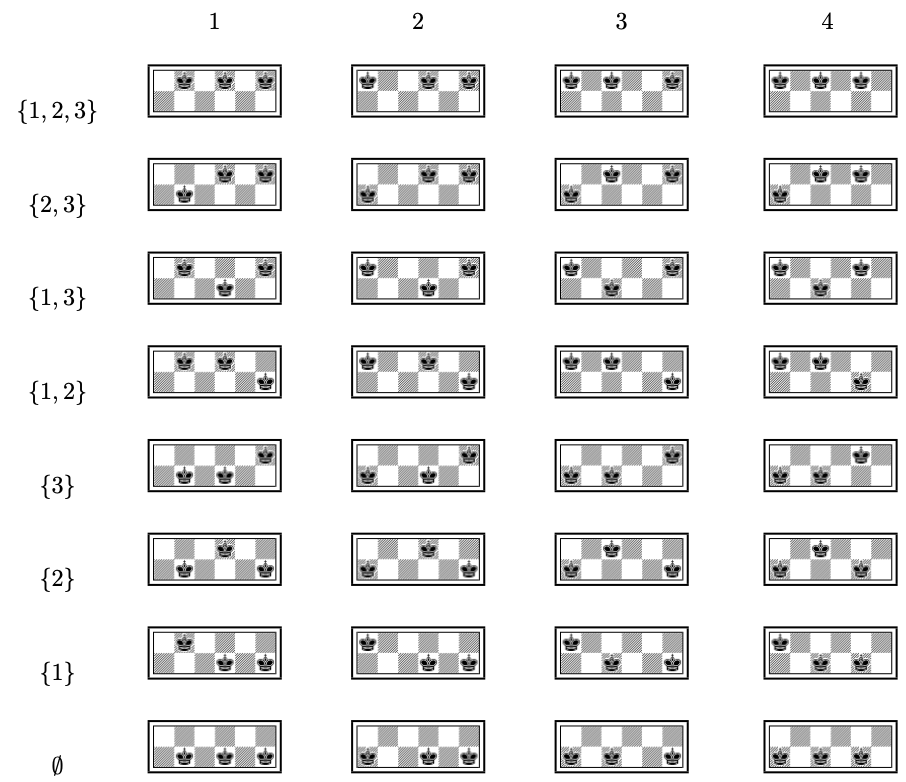}}
\end{center}
\caption{Entries in the matrix $M_6$}
\label{3matrix}
\end{figure}

\begin{lemma}
Every independent arrangement of $n$ kings on a $2\times 2n$ rectangle appears exactly once as an entry in the matrix $M_{2n}$.
\end{lemma}
 
\begin{proof} 
By construction, the set of matrix entries is contained in the set of independent $n$-kings arrangements on $2\times 2n$ rectangles.  Further the matrix entries encompass all ways to places kings in the top or bottom row of each rectangle and all the ways to place kings in left or right columns so that all right kings appear to the left of all left kings.  Any arrangement that does not appear as a matrix entry would necessarily then have a right king appearing to the right of a left king, but this contradicts independence. Thus we have arrangement all possible arrangements of non-attacking kings on a $2\times 2n$ rectangle appearing in exactly one entry of the $2^n \times (n+1)$ matrix $M_{2n}$.  
\end{proof}

Our goal now is to vertically concatenate these rectangular arrangements to form allowable independent solutions.  The first observation is that for such a concatenation to be independent, the row index of the upper rectangle, must contain the row index of the lower rectangle.

\begin{lemma}\label{lemma_BsubA}
Given an independent arrangement of $n$ kings on the rectangle indexed by $(A,k)$ where $A\subseteq [n]$ and $1\leq k \leq n+1$, any independent arrangement on another $2\times 2n$ rectangle that can be concatenated below to form an independent arrangement of size $4\times 2n$ must be indexed by $B$ where $B\subseteq A$.
\end{lemma}

\begin{proof}To the contrary, if $B \nsubseteq A$, then there exists an index $i$ that is in $B$ and is not in $A$.  This indicates a king in the top row of the strip indexed by $B$ and a king in the bottom row of the strip indexed by $A$ in square $i$.  No matter the placement, left or right, for each king, the two kings may attack.
\end{proof}

This condition is necessary but not sufficient.  For example the strip in position $(\{2\}, 2)$ may not be placed independently under the strip indexed by $(\{1,2\}, 4)$ despite $\{2\} \subseteq \{1,2\}$.  This is because the king in the top row of strip $(\{2\},2)$ in square $2$ has been pushed to the right, so it conflicts with the king in the bottom row of the strip $(\{1,2\},4)$ in square $3$.  We need a condition on the left and right positions of the kings.  To find this, we define two functions dependent on $A$, $B$, and $k$.

\begin{definition}
\begin{enumerate}
\item Let $p(A,B,k) = \max \{ \{i\in [n+1] : i<k , i \notin A,\mbox{ and }i-1 \in B\} \cup \{1\}\}$.
\item Let $q(A,B,k)= \min \{ \{ i \in [n+1]: i >k, i \in B, \mbox{ and } i-1 \notin A\} \cup \{n+1\}\}$
\end{enumerate}
\end{definition}

\begin{proposition}\label{prop_pandq}
Given a $2\times 2n$ arrangement of $n$ independent kings indexed by $(A,k)$, the $2\times 2n$ independent arrangements that maybe independently concatenated below $(A,k)$ are exactly those which are indexed by $(B,i)$ where $B\subseteq A$ and $p(A,B,k) \leq i \leq q(A,B,k)$.
\end{proposition}

\begin{proof}
Lemma~\ref{lemma_BsubA} forces $B\subseteq A$.  Then, note that locally for all $1< j \leq n+1$ the concatenation of two strips indexed by $A$ and $B$ will never have a king in square $j$ of $(A,k)$ and a king in square of $j$ of $(B,k)$ attacking each other, because either both kings are in the top row, both are in the bottom row, or the king in $(A,k)$ is in the top row while the king in $(B,i)$ is in the bottom row.  There are however two ways for a concatenation of rectangular strips $(A,k)$ and $(B,i)$ to have attacking kings in squares indexed some $j-1$ and $j$; either there is a right king in the top row of square $j-1$ in $(B,i)$ and a left king in the bottom row of square $j$ in $(A,k)$, or there is a right king in the bottom row of square $j-1$ in $(A,k)$ and a left king in the top row of square $j$ in $(B,i)$.  (See Figure~\ref{conflict}.)

\begin{figure}[ht]
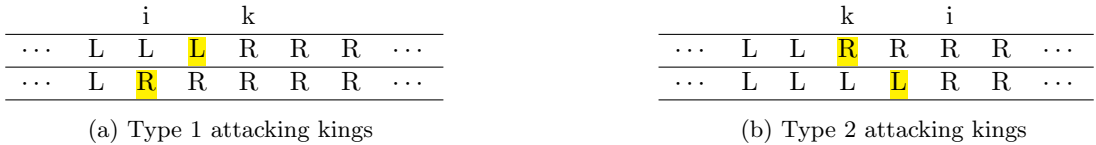

\begin{subfigure}{.5\textwidth}
\hspace{0.4in}
\begin{tabular}{cccccccc}
&&i&&k&&&\\
\hline
$\cdots$ &L&L&\hl{L}&R&R&R&$\cdots$\\
\hline
$\cdots$ &L&\hl{R}&R&R&R&R&$\cdots$\\
\hline
\end{tabular}
\caption{Type 1 attacking kings}\label{conflict1}
\end{subfigure}
\quad
\begin{subfigure}{.5\textwidth}
\hspace{0.4in}
\begin{tabular}{cccccccc}
&&&k&&i&&\\
\hline
$\cdots$ &L&L&\hl{R}&R&R&R&$\cdots$\\
\hline
$\cdots$ &L&L&L&\hl{L}&R&R&$\cdots$\\
\hline
\end{tabular}
\caption{Type 2 attacking kings}\label{conflict2}
\end{subfigure}
\caption{Possible attacks by arrangements of independent kings on two $2\times 2n$ rectangles}\label{conflict}
\end{figure}

Let us first consider how we can avoid the situation of an attack by a left king in the fixed strip $(A,k)$ and a right king in strip $(B,i)$ for some $i$.  We note, the largest value for $i$ that could generate a pair of attacking kings is $i=k-2$, as we can see from Figure~\ref{conflict1}, because for any larger values there cannot be a left king in $(A,k)$ that is immediately to the left of a right king in $(B,i)$.  So we will have an attacking pair in this case if the king in $(B,i)$ is in the top row and the king in $(A,k)$ is in the bottom row, that is, if
\begin{enumerate}[i.]
\item $i< k-1$, 
\item $i \in B$, and
\item  $i+1 \notin A$. 
\end{enumerate} 
In particular,  we look to the largest index over all values $i$ satisfying the criteria above, that is, over all $i$ where there is a pair of attacking kings of Type 1, because every index less than or to the left of this maximum would still have the attacking pair of kings with one king in the square indexed by that maximum and the other directly to the right.   This is because moving to the left in the matrix will only increase the number of right kings in the arrangement indexed by $B$.  We can then choose the index which is one greater than the maximum as the starting value for which all strips indexed by $B$ and any index greater than or equal to this maximum plus one will not have a Type 1 conflict.  If no such value exists, then all arrangements in row $B$ will not have a conflict with the arrangement $(A,k)$ and we set our lower bound to $1$. Using a change of variable where we send all $i+1$ to $i$, we can say the lower bound for non-attacking arrangements is:
\begin{equation*}
p(A,B,k) = \max \{ \{i\in [n+1] : i<k , i \notin A,\mbox{ and }i-1 \in B\} \cup \{1\}\}.
\end{equation*}

Similarly, we need to look at the situation of an attack by a right king in $(A,k)$ by a left king in $(B,i)$ for some $i$.  In Figure~\ref{conflict2}, we can see that the smallest value for an attacking pair would be $i=k+2$.  So to have an attacking pair of the second type, we need
\begin{enumerate}[i.]
\item  $i>k+1$, 
\item $i-1\in B$, and 
\item $i-2 \notin A$. 
\end{enumerate}
 After determining the minimum of these values $i$ where have pairs of attacking kings of Type 2, we set the upper bound to the minimum minus one as the last possible non-attacking arrangement when we move from left to the right across a row of $M_{2n}$ increasing the number of queens in left columns.  If no such value exists, then all arrangements of $B$ have no attacking pairs of the second type, and we can set the upper bound to $n+1$.  Again applying a change of variable sending all $i-1$ to $i$, we have
\begin{equation*}
q(A,B,k)= \min \{ \{ i\in [n+1]:i >k,  i \in B, \mbox{ and } i-1 \notin A\} \cup \{n+1\}\}.
\end{equation*}
\end{proof}

We illustrate these functions with the following example.

\begin{example}\label{ex_pandq}
Let $n=7$.  Let $A=\{1,2,5,7,8\}$,  $B= \{1,2,5,7\}$ and $k=4$.  We compute the upper and lower bounds.
\begin{align*}
p(A,B,4) &=  \max \{ \{i :i \notin A, i<4 \mbox{ and }i-1 \in B\} \cup \{1\}\}\\
&= \max \{ \{3\} \cup \{1\}\} = 3\\
q(A,B,4) &= \min \{ \{ i : i \in B, i >4 \mbox{ and } i-1 \notin A\} \cup \{n+1\}\}\\
&=\min \{\{5,7\} \cup \{8\}\}=5.
\end{align*}
Thus arrangements indexed by $B$ that may be concatenated below $(A,4)$ are exactly $(B,3)$, $(B,4)$, and $(B,5)$.  See Figure~\ref{example}.

\begin{figure}[ht]
%
%
%
%
\begin{center}
\scalebox{0.9}{\includegraphics{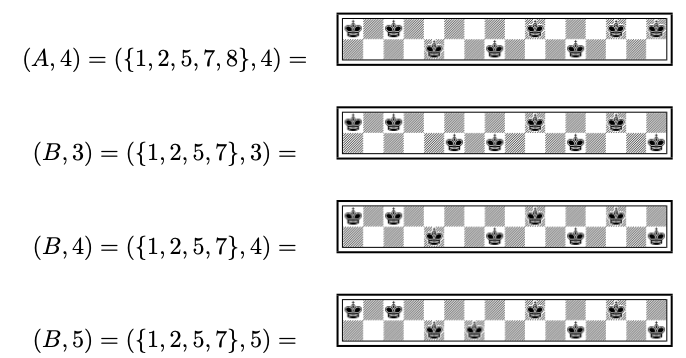}}
\end{center}
\caption{Example of independent arrangements $(\{1,2,5,7\},i)$ which may be concatenated independently below the independent arrangement $(\{1,2,5,7,8\},4)$}
\label{example}
\end{figure}
\end{example}

Now we are ready for the main result.

\begin{theorem}\label{theorem_kings}
Define a sequence of matrices $M^\ell$ with entries $\{m_{A,k}^{\ell}\}$ for $1\leq \ell \leq n$, $A\subseteq [n]$, and $1\leq k \leq n+1$ such that $m^1_{A,k} = 1$ for all pairs $(A,k)$ and for $1<\ell \leq n$
\[ m^{\ell}_{A,k} = \sum_{B\subseteq A} \sum_{i=p_(A,B,k)}^{q_(A,B,k)} m^{\ell -1}_{B,i}. \]

Then $K(n)$, the number of maximum arrangements of nonattacking kings on the $2n\times 2n$ chessboard, is
 \[K(n) = \sum_{A\subseteq [n]} \sum_{k=1}^ {n+1} m^{n}_{A,k}. \]
\end{theorem}

\begin{proof}
The sum of the entries of the matrix $M^1$ gives the number of nonattacking arrangements of $n$ kings on a $2\times 2n$ strip as each of the entries corresponds bijectively with a specific strip indexed by a subset $A\subseteq [n]$ and a value $1\leq k \leq n+1$.  Because $p(A,B,k)$ and $q(A,B,k)$ provide bounds for all valid arrangement which may be independently concatenated below $(A,k)$, by Proposition~\ref{prop_pandq} the double sum
\[m^2_{A,k}=\sum_{B\subseteq A} \sum_{i=p_(A,B,k)}^{q_(A,B,k)} m^{1}_{B,i}\]
counts all valid $4\times 2n$ maximum arrangements of nonattacking kings where the top $2\times 2n$ arrangement is $(A,k)$, that is, the entries $m_{A,k}^2$ correspond to the number of independent arrangements on $2\times 2n$ rectangles which may be concatenated below the arrangement $(A,k)$ while preserving independence.  Their sum thus is the total number of nonattacking arrangements of $2n$ kings on a $4\times 2n$ strip.  As this process is repeated, we see that a $2\ell \times 2n$ independent arrangement is found by concatenating a single independent arrangement on a $2\times 2n$ rectangle $(A,k)$ above and existing independent arrangement of the $2(\ell-1)\times 2n$ strip so that independence is also preserved along this concatenation.  Inductively, we know the number of $2(\ell-1)\times 2n$ arrangements with arrangement $(B,i)$ at the top is $m_{B,i}^{\ell-1}$, and as in the case where $\ell =2$ we utilize the upper and lower bound functions $p$ and $q$ to sum over precisely the arrangements which can be concatenated independently
\end{proof}

\begin{example}
Let $n=3$.

\[ m^1 = \begin{bmatrix}
1&1&1&1\\
1&1&1&1\\
1&1&1&1\\
1&1&1&1\\
1&1&1&1\\
1&1&1&1\\
1&1&1&1\\
1&1&1&1\\
\end{bmatrix}
\quad
 m^2 = \begin{bmatrix}
 32&32&32&32\\
 12&16&16&16\\
 14&14&14&14\\
 16&16&16&12\\
 7&7&8&8&\\
 6&8&8&6\\
 8&8&7&7\\
 4&4&4&4\\
 \end{bmatrix}
 \quad
 m^3 = \begin{bmatrix}
408&408&408&408\\
88&134&134&134\\
110&110&110&110\\
134&134&134&88\\
38&38&46&46\\
30&44&44&30\\
46&46&38&38\\
16&16&16&16\\
\end{bmatrix}
\]
 \[K(3) = \sum_{A\subseteq [3]} \sum_{k=1}^ {4} m^{3}_{A,k}=3600. \]
\end{example}

Of course, if we want to enumerate arrangements on a $2m \times 2n$ rectangular chessboard with $m\leq n$, we can simply utilize matrices $M^m$. We have the following corollary:

\begin{corollary}
The number of maximum arrangements of nonattacking kings on a $2m\times 2n$ rectangle for $m\leq n$ is given by the sum
\[K(n,m)= \sum_{A\subseteq [n]} \sum_{k=1}^{n+1} m^{m}_{A,k}.\]
\end{corollary}

Finally, Theorem~\ref{theorem_kings} provides the following algorithm for computing the values of $K(n)$.
\begin{algorithm}\label{algorithm}
\begin{enumerate}
\item For every triple $(A,B,k)$ such that $B\subseteq A \subseteq [n]$ and $1\leq k \leq n+1$, compute $p(A,B,k)$ and $q(A,B,k)$.
\item Define the matrix $M^1$ with rows indexed by subsets of $[n]$ and columns indexed by $1\leq k \leq n+1$ where each entry is set to 1.
\item For $2\leq \ell \leq n$, determine the entries in the matrix $M_{\ell}$ by summing over entries $m_{B,i}^{\ell -1}$ in the matrix $M^{\ell -1}$ where $B\subseteq A$ and $p(A,B,k) \leq i \leq q(A,B,k)$.
\item Sum the entries of $M^n$.
\end{enumerate}
\end{algorithm}

\section{Complexity}

We can estimate the computational complexity of Algorithm~\ref{algorithm}.   It is computationally cheap to determine the values of $p(A,B,k)$ and $q(A,B,k)$. Each of these functions run in polynomial time with bound $\mathcal{O}(n^2)$, that is, $2n+1$ ways to check the requirements, multiplied by $n$ for all possible values of $i$, plus $n$ to determine the minimum.  However, the difficulty is that the algorithm sums over all pairs of subsets $B\subseteq A \subseteq [n]$ in $\sum_{i=0}^n {n\choose i} 2^i = 3^n$ ways, providing an overall upper bound on the order of  $\mathcal{O}(3^n n^3)$.  We note, many of the entries in each matrix $M$ are repeated because of symmetry; for example columns indexed by $i$ and $n+2-i$ are equal, so some further modifications to this bound can be made.

It is of interest to compare this complexity to previous work.  As mentioned in Section~\ref{sec_introduction}, both Knuth and Wilf provided algorithms for computing the number of maximal arrangements on a $2m\times 2n$ chessboard.  First, Knuth proposes creating directed graph with vertices $(a,b,S)$ such that $0\leq a \leq m$, $0\leq b \leq n$, and $S\subseteq [m]$ where there is a directed edge from $(a, b,S)$ to $(a', b+1, T)$ if and only if $S\subset T$ and the following conditions are satisfied:
\begin{enumerate}[i.]
\item If $a< j <a'$ and $j\in S$, then $j+1\in T$.
\item If $a> j >a'$ and $j+1\in S$, then $j\in T$.
\end{enumerate}
Once the directed graph is established, one counts the number of directed paths of length $n$ to determine the number of maximal independent solutions.  One way to establish an upper bound on time complexity is to utilize the adjacency matrix for this directed graph.  After finding the matrix, we may raise it to the $n$th power to find the number of directed paths between pairs of vertices of length $n$.  Note, the $b$ in a triple describing a vertex is simply a placeholder as all paths of length $n$ will have the same sequence of middle coordinates.  Thus we look at pairs of vertices $(a, S)$ and $(a',T)$ where $S\subset T \subseteq [m]$.  For each of the $m^2 3^m$ pairs of vertices we check the condition given by Knuth and create the $(m+1)2^m \times (m+1)2^m$ adjacency matrix.  We then apply $n-1$ repeated matrix multiplications using the Coppersmith-Winograd~\cite{CoppersmithWinograd} algorithm to find the bound of $\mathcal{O}\left((n-1)(m+1)^{2.376} \cdot 2^{2.376m}\right)$ for this part of the process before summing the entries of the final matrix.  For the complete algorithm, if $m=n$ we have $\mathcal{O}(n^{5.376} \cdot 3^n \cdot  5.2^{n}) =\mathcal{O} (n^{5.376} \cdot 15.6^n)$.  Shortcuts may be found to improve this bound especially by utilizing symmetry to halve the calculations.

Wilf's strategy is also to partition the chessboard; in this case, into columns of size $2m\times 2$.  To do this, the set of $(m+1)2^m$ maximal arrangements on a $2m\times 2$ column indexes the rows and columns of a square matrix.  This matrix, $\Lambda_m$, called the transfer matrix, is block upper triangular and has entries which are either zero or one.  The number of maximal independent arrangements of kings for the $2m\times 2n$ chessboard is given by summing the entries of $\Lambda_{m}^{n-1}$.  After constructing the unique transfer matrix and applying $n-2$ repeated matrix multiplications, we have for $m=n$ a bound of $\mathcal{O} \left((n-2)(n+1)^{2.376} 2^{2.376n}\right) = \mathcal{O} (n^{3.376} \cdot 5.2^n)$.  In either case, Algorithm~\ref{algorithm} is an improvement on previous methods.  A Mathematica~\cite{Mathematica} program that uses Algorithm~\ref{algorithm} to calculate the values of $K(n)$ is available upon request.


\newcommand{\book}[4]{{\sc #1,} #2, #3 (#4)}
\newcommand{\preprint}[3]{{\sc #1,} #2, preprint #3.}

\end{document}